\documentclass[10pt,a4paper]{article}
\usepackage{amsmath,amsfonts,amssymb,amsthm,epsfig,epstopdf,titling,url,array}
\usepackage[utf8]{inputenc}
\usepackage{longtable}
\usepackage{setspace}
\usepackage{geometry}                		                		
\usepackage{graphicx}	
\usepackage{subfigure}
\usepackage{float}	
\usepackage{xcolor}
\usepackage{comment}
\usepackage{url}
\usepackage{hyperref}

\usepackage[english]{babel}

\theoremstyle{definition}
\newtheorem{defi}{Definition}
\newtheorem{ex}[defi]{Example}

\theoremstyle{remark}

\newtheorem{rem}[defi]{Remark}

\theoremstyle{plain}
\newtheorem{thm}[defi]{Theorem}
\newtheorem*{B}{Baire Category Theorem}

\newtheorem{lem}[defi]{Lemma}

\newtheorem{cor}[defi]{Corollary}

\title{Nowhere dense competing holes in open dynamical systems}
\author{F. Ciavattini$^*$, T.H. Steele$^{**}$}
\date{\small{$^*$Scuola di Scienze e Tecnologie, Università degli Studi di Camerino, Italy, e-mail: filippo.ciavattini@studenti.unicam.it}\\
$^{**}$ Department of Mathematics, Weber State University,  Ogden UT,  84408, USA, e-mail: thsteele@weber.edu}

\begin{document}
\maketitle
\begin{abstract}
Let $\mathcal{M}$ be a compact metric space with no isolated points, and $f:\mathcal{M}\longrightarrow\mathcal{M}$ a homeomorphism. Consider a sequence of shrinking open balls $\{B^i_n\}_{n\in\mathbb{N}}^{i\in\mathbb{N}}$ with centers $\{p_i\}_{i=1}^\infty\subseteq\mathcal{M}$ and radii $\{\rho^i_n\}_{n=1}^\infty$. For every point $x\in\mathcal{M}$ and $n\in\mathbb{N}$, consider which ball the trajectory $\{x,f(x),f^2(x),\dots\}$ of the point first visits. We find that whenever the closure of $\{p_i\}_{i=1}^\infty$ is nowhere dense, and with very minor restrictions on $\{\rho_n^i\}_{n\in\mathbb{N}}^{i\in\mathbb{N}}$, the typical trajectory $\{f^k(x)\}_{k=0}^\infty$ will first visit, for each $i$, the ball $B^i_n$, for infinitely many $n$. This is never the case, should $\{p_i\}_{i=1}^\infty$ be somewhere dense.\\


\end{abstract}

\section{Introduction}
Suppose that $\mathcal{M}$ is a compact metric space, and $f:\mathcal{M}\longrightarrow\mathcal{M}$ is transitive. Take $B_1,B_2,\dots,B_m$ to be disjoint open balls in $\mathcal{M}$. If $x$ is a point in $\mathcal{M}$ for which the trajectory $\mathcal{O}^+(x)=\{x,f(x),f^2(x),\dots\}$ is dense, then, regardless of the size of the balls $B_i$, there will be a point $f^k(x)$ in $\mathcal{O}(x)$ that will fall into one of the balls. In \cite{dfr}, the authors consider not just one set of the balls $\{B_i\}_{i=1}^m$, but a sequence of shrinking balls. In particular, let $\{p_i\}_{i=1}^m$ be a set of points in $\mathcal{M}$, and $\{\rho^i_n\}_{n\in\mathbb{N}}^{1\leq i\leq m}$ a collection of decreasing sequences such that, for each $i$, $\rho^i_{n+1}<\rho^i_n$ for all $n$, and $\lim_{n\longrightarrow\infty}\rho^i_n=0$.
For each $n$, we now have a new set of balls $\{B_n^1,B_n^2,\dots,B_n^m\}$, where $B^i_n=B_{\rho^i_n}(p_i)$. Again, for each $n$, the dense trajectory $\mathcal{O}(x)$ will eventually visit one of the balls $\{B_n^i\}^{1\leq i\leq m}$. But as $n$ changes, the first ball $B^i_n$ that is visited by the trajectory $\mathcal{O}(x)$ could also change.
 The main result of \cite{dfr} is surprising. The authors establish the existence of
 \begin{enumerate}
 	\item a residual set $\mathcal{A}\subseteq\mathcal{M}^m$, with each element of $\mathcal{A}$ providing centers $\{p_1,p_2,\dots,p_m\}$, and
 	\item a residual set $\mathcal{B}$ in $\mathcal{M}$, each of whose elements $x$ has a dense trajectory $\mathcal{O}(x)$ in $\mathcal{M}$, such that for each $1\leq i \leq m$, the trajectory $\mathcal{O}(x)$ will first visit $B^i_n$, for infinitely many values $n$.  \end{enumerate}
Moreover, this is true regardless of the rate at which the radii shrink to zero. The only condition placed on the collection of radii $\{\rho^i_n\}_{n\in\mathbb{N}}^{1\leq i\leq m}$ is that $\overline{B^i_1}\cap\overline{B^j_1}=\emptyset$, whenever $i\neq j$. This is a natural restriction to make, as it insures that the closure of the balls $B^i_n$, $1\leq i\leq m$, are disjoint for each $n$. Observe that for any collection of sequences $\{a_n^i\}_{n\in\mathbb{N}}^{1\leq i\leq m}$, such that for each $i$, $a_{n+1}^i<a^i_n$ for all $n$, and $\lim_{n\longrightarrow\infty}a^i_n=0$, an appropriate collection of radii $\{\rho^i_n\}_{n\in\mathbb{N}}^{1\leq i\leq m}$ can be obtained by considering tails of the sequences $\{a_n^i\}_{n\in\mathbb{N}}^{1\leq i\leq m}$: for each $1\leq i\leq m$, there exists an $N(i)$ such that $\overline{B^i_{N(i)}}\cap\overline{B^j_{N(j)}}=\emptyset$, whenever $i\neq j$. Let $\rho^i_k=a^i_{N(i)+k}$.

The results of \cite{dfr} are generalized to the case of countably many balls in \cite{fs}. Here, the authors consider a sequence of centers $\{p_i\}_{i=1}
^\infty$ such that the derived set $D(\{p_i\}_{i=1}
^\infty)=\bigcap_{k=1}^\infty\overline{\bigcup_{i>k}p_i}$ is disjoint from the sequence of centers $\{p_i\}_{i=1}
^\infty$. Again, by considering appropriate tails of any collection of infinitesimal sequence $\{a_n^i\}_{n\in\mathbb{N}}^{i\in\mathbb{N}}$, one can insure that $\overline{B^i_1}\cap\overline{B^j_1}=\emptyset$, whenever $i\neq j$. Should $f:\mathcal{M}\longrightarrow\mathcal{M}$ be a transitive homeomorphism, then there exist
\begin{enumerate}
	\item a residual set $\mathcal{A}\subseteq\ell_\infty(\mathcal{M})$, and
	\item a residual set $\mathcal{B}$ in $\mathcal{M}$, each of whose elements $x$ has a dense trajectory $\mathcal{O}(x)$ in $\mathcal{M}$
\end{enumerate}
such that for each $\{p_1,p_2,\dots\}$ in $\mathcal{A}$, and each $i$ in $\mathbb{N}$, the trajectory $\mathcal{O}(x)$ will first visit the ball $B^i_n$ centered at $p_i$, for infinitely many values $n$. As in the finite case, this is regardless of the rate at which the radii $\rho^i_n$ shrink to zero.

Here, we extend the result of \cite{dfr} and \cite{fs} by no longer requiring that the sequence of centers $\{p_1,p_2,\dots\}$ is disjoint form its derived set. Three cases are considered:
\begin{enumerate}
	\item the closure of $\{p_i\}_{i=1}
	^\infty$ is countable;
	\item the closure of $\{p_i\}_{i=1}
	^\infty$ is nowhere dense, and
	\item the closure of $\{p_i\}_{i=1}
	^\infty$ contains an open ball in $\mathcal{M}$.
\end{enumerate}
We show that the results of \cite{dfr} and \cite{fs} generalize to cases 1 and 2, for appropriate tails of any collection of infinitesimal sequences $\{a^i_n\}_{n\in\mathbb{N}}^{i\in\mathbb{N}}$. In the case that the closure of $\{p_i\}_{i=1}
^\infty$ is somewhere dense, there is no possible result analogous to those of \cite{dfr} and \cite{fs}, for any collection of infinitesimal sequences.

We proceed through several sections. In Section \ref{preliminaries}, we present the notation, definitions and previously known results necessary to our analysis. The main results come in Section \ref{count-closure}, where we consider the case that the closure of centers $\{p_i\}_{i=1}
^\infty$ is countable. It is then easy, with Section \ref{nwd-closure}, to get the desired results in the case that the closure of the $\{p_i\}_{i=1}
^\infty$ is nowhere dense. The brief and final Section \ref{swd_closure} addresses the case that the closure of $\{p_i\}_{i=1}
^\infty$ is somewhere dense.

\section{Preliminaries}\label{preliminaries}
\subsection{Notation and definitions}
Throughout, let $(\mathcal{M},d)$ be a compact metric space with no isolated points, while $f:\mathcal{M}\longrightarrow\mathcal{M}$ is a homeomorphism. Let $\mathbb{N}$ denote the set of positive integers and $\mathbb{N}_0=\mathbb{N}\cup\{0\}$. Take $\mathcal{O}^-(x)$ and $\mathcal{O}^+(x)$ to be, respectively, the backward and forward $f-$orbit of the point $x$; that is, $\mathcal{O}^-(x)=\{f^k(x):k\in\mathbb{Z}\setminus\mathbb{N}_0\}$ and $\mathcal{O}^+(x)=\{f^k(x):k\in\mathbb{N}_0\}$.

Given a point $p\in\mathcal{M}$ and a closed subset $F\subseteq\mathcal{M}$, let
$$d(p,F)=\min\{d(p,q):q\in F\}\text{.}$$
Thus, if $p\notin F$, then $d(p,D)>0$. Given another closed subset $G\subseteq \mathcal{M}$, let $d_H(F,G)$ be the Hausdorff distance between $F$ and $G$, so that
$$d_H(F,G)=\inf\{\epsilon>0:F\subseteq B_\epsilon(G)\text{ and }G\subseteq B_\epsilon(F)\}\text{.}$$

In the following we will consider $\{p_i\}_{i=1}^\infty$, a sequence of pairwise distinct points of $\mathcal{M}$.
For every $i\in\mathbb{N}$, let $\{\rho^i_n\}_{n=1}^\infty$ be a decreasing sequence such that$$\lim_{n\longrightarrow\infty}\rho^i_n=0\text{.}$$

The open ball with center $p_i$ and radius $\rho^i_n$ will be denoted by $B^i_n$; that is, $B^i_n=B_{\rho^i_n}(p_i)$.

Re-introducing the notation used in \cite{dfr}, by \emph{open dynamical system} we mean a 4-tuple $(\mathcal{M},d,f,\mathfrak{S})$, where $\mathfrak{S}:=\{p_i, \{\rho^i_n\},  i\in \mathbb{N}, n\in\mathbb{N}\}$. In the sequel we study the predictability of open dynamical systems.

Let us recall the definition of several subsets used in \cite{fs}.
\begin{defi}\label{defT11}
	
	In the following, we adopt the convention $\min \emptyset = +\infty$.
	
	\begin{itemize}
		\item For $i\in \mathbb{N}$, a point $x\in \mathcal{M}$ belongs, respectively,  to $\mathfrak{T}^1(i)$ and to $\mathfrak{T}^1$, if there exists $c\in\mathbb{N}_0$ such that 
		\begin{equation*}\label{eq_defT1i}
			\mathcal{O}^+(x)\cap  B^i_{c}=\emptyset \quad\text{and}\quad \mathcal{O}^+(x)\cap  \bigcup_{j=1}^\infty B^j_{c}=\emptyset.
		\end{equation*}
		Thus, $x\in\mathfrak{T}^1$ if $\mathcal{O}^+(x)$ misses all of the balls $B^i_n$, whenever $n\geq c$.
			
		\item A point $x\in \mathcal{M}$ belongs to $\mathfrak{T}^2(i)$ if there is $c\in\mathbb{N}_0$ such that, $\text{for all } n\geq c$,
		\begin{equation*}
			\min\{k\in\mathbb{N}_0 : f^k(x)\in B^i_n\}<\min\{k\in\mathbb{N}_0 : f^k(x)\in B^j_n \ \text{ for all}\ j\ne i\}.
			\label{t2}
		\end{equation*}
		\noindent We set $\mathfrak{T}^2:=\bigcup_{j=1}^\infty \mathfrak{T}^2(j)$. Thus, $x\in\mathfrak{T}^2$ if $\mathcal{O}^+(x)$ always first visits $B^i_n$, for some fixed $i$, whenever $n\geq c$. 
	\end{itemize}
	
	The points in $\mathfrak{T}^1\cup \mathfrak{T}^2$ are called \emph{decisive points}.
	\begin{itemize}
		
		\item A point $x\in \mathcal{M}$ belongs to $\mathfrak{T}(i)$ if 
		the following inequality holds for infinitely many natural numbers $n$:
		$$\min\{k\in\mathbb{N}_0:f^k(x)\in B^{i}_n\}<\min\{k\in\mathbb{N}_0:f^k(x)\in B^j_n\  \text{ for all } j\ne i\}.$$ 
		
		\item  A point $x\in \mathcal{M}$ belongs to $\mathfrak{T}^3$ if and only if there are $i_1 \ne i_2$ in $\mathbb{N}$ such that $x\in\mathfrak{T}(i_1)\cap\mathfrak{T}(i_2)$.  
	\end{itemize}
	
	The points in $\mathfrak{T}^3$ are called \emph{indecisive points}.
	
	\begin{itemize}
		\item We say that $x$ belongs to $\mathfrak{T}$ if
		\begin{equation*}\label{def_compl_indec}
			x\in \bigcap_{j=1}^\infty \mathfrak{T}(j).
		\end{equation*}
	\end{itemize}
	
	We will call the points in $\mathfrak{T}$ \emph{completely indecisive}.
	
\end{defi}

\begin{defi}
	We say that the open dynamical system $(\mathcal{M},d,f,\mathfrak{S}:=\{p_i, \{\rho^i_n\},  i\in \mathbb{N}, n\in\mathbb{N}\})$ is \emph{completely indecisive} if the set $\mathfrak{T}\ \text{\emph{is residual.}}$
\end{defi}

We find that
\begin{enumerate}
	\item there is a dense set $\mathcal{S}\subseteq\ell_\infty(\mathcal{M})$, each of whose elements has a countable closure, and for which we can find a sequence of radii such that our system is completely indecisive;
	
	\item there is a residual set $\mathcal{T}\subseteq\ell_\infty(\mathcal{M})$, each of whose elements has a nowhere dense closure, and for which we can find a sequence of radii such that our system is completely indecisive, and
	
	\item for any sequence $\{p_i\}_{i=1}^\infty$ such that the closure Cl$(\{p_i\}_{i=1}^\infty)$ is somewhere dense, there is not a sequence of radii such that our system is completely indecisive.
\end{enumerate}

Given that we will be working within complete metric spaces, we will be able to make good use of the Baire category theorem.

\begin{defi}
	\noindent A subset $S$ of a topological space $X$ is said to be of \emph{ the first category} if there exists a countable family $\{S_i\}_{i\in\mathbb{N}}$ of nowhere dense subsets of $X$ such that $S=\cup_{i\in\mathbb{N}} S_i$. We say that a set $S\subseteq X$ is \emph{residual} if $X\setminus S$ is of the first category. An element of a residual subset of $X$ is called either a \emph{typical} or a \emph{generic} element of $X$.
\end{defi}

\begin{B}\label{th_baire}
	If $X$ is a complete metric space, and $S$ is a first category subset of $X$, then $X\setminus S$ is dense.
\end{B}

\section{Cl$(\{p_i\}_{i=1}^\infty)$ is countable}\label{count-closure}
In this section we will consider sequences $\{p_i\}_{i=1}^\infty$ such that $D(\{p_i\}_{i=1}^\infty)$ --- the derived set of $\{p_i\}_{i=1}^\infty$ --- is countable.
Given a set $A\subseteq\mathcal{M}$, let $D(A)$ denote its derived set. Thus, for any ordinal $\alpha$, we can define, as done in \cite{kuratowski2014topology}, the $\alpha$-th derived set $D^\alpha(A)$ in the following way:
\begin{itemize}
	\item $D^0(A)=A$;
	\item $D^\alpha (A)=D(D^{\alpha-1}(A))$, if $\alpha$ is not a limit ordinal, and
	\item  $D^\alpha (A)=\bigcap_{\lambda<\alpha}D^\lambda (A)$, if $\alpha$ is a limit ordinal. 
\end{itemize}
In other words, we will consider the sequences $\{p_i\}_{i=1}^\infty$ such that there exists an ordinal $\beta<\Omega$ --- where $\Omega$ is taken to be the first uncountable ordinal --- such that $D^{\beta+1}(\{p_i\}_{i=1}^\infty)=\emptyset$.

\subsection{Preliminary results} \label{prevresults}

Let $\mathcal{S}$ be the set of closed sequences; that is, $\mathcal{S}=\{\{p_i\}_{i=1}^\infty:\text{Cl}(\{p_i\}_{i=1}^\infty)=\{p_i\}_{i=1}^\infty\}$ in $\ell_\infty(\mathcal{M})$. Here, we take the set $\ell_\infty(\mathcal{M})$ to be the collection of the sequences in $\mathcal{M}$ with the supremum metric given by $\rho(\{p_i\}_{i=1}^\infty,\{q_i\}_{i=1}^\infty)=\sup_id(p_i,q_i)$.
One might suspect that $\mathcal{S}$ is a closed set, but, as the following example shows, the set $\mathcal{S}$ is not closed in $\ell_\infty(\mathcal{M})$.
\begin{ex}\label{example}
	Let $\mathcal{M}=[0,1]$. Take two sequences $\{a_n\}_{n=1}^\infty,\{b^i\}_{i=1}^\infty\subseteq(0,1]$ such that $a_n\xrightarrow{n\longrightarrow\infty}0$ and $b^i\xrightarrow{i\longrightarrow\infty}0$. Let us define$$p^i_n=
	\begin{cases}
		a_n+b^i & n\neq i\\
		b^i & n=i
	\end{cases}$$
	It follows that, for every $i\in\mathbb{N}$, $p^i_n\xrightarrow{n\longrightarrow\infty}b^i\in\{p^i_n\}_{n=1}^\infty$. Thus, for every $i\in\mathbb{N}$, the sequence  $\{p^i_n\}_{n=1}^\infty$is an element of $\mathcal{S}$. Since,$$\rho(\{p^i_n\}_{n=1}^\infty,\{a_n\}_{n=1}^\infty)=\max\{b^i,\lvert b^i-a^i\rvert\}\xrightarrow{i\longrightarrow\infty}0\text{,}$$we have $\{p^i_n\}_{n=1}^\infty\xrightarrow{i\longrightarrow\infty}\{a_n\}_{n=1}^\infty$. But, since $a_n\xrightarrow{n\longrightarrow\infty}0$ and $a_n>0$ for every $n\in\mathbb{N}$, we also have $\{a_n\}_{n=1}^\infty\notin\mathcal{S}$. Therefore $\mathcal{S}$ is not closed.
\end{ex} 

We now make several observation concerning elements of $\mathcal{S}$. In particular, we develop two classes of elements in $\mathcal{S}\subseteq\ell_\infty(\mathcal{M})$ that are both dense, and fundamental to our analysis of those sequences with a countable closure.

Let $\mathcal{I}=\{x\in\mathcal{M}:\overline{\mathcal{O}^-(x)}=\mathcal{M}\}$; as shown in \cite{kurka2003topological}, $\mathcal{I}$ is residual. Thus, given an arbitrary countable collection of point $\{y_i\}_{i=1}^\infty\subseteq\mathcal{M}$, the set $\mathcal{I}\setminus\bigcup_{i=1}^\infty\mathcal{O}(y_i)$ is dense, because the set $\mathcal{O}(y_i)$ is of first category for every $i\in\mathbb{N}$.
	
Let $\mathcal{S}^1$ be the set of sequences $\{p_i\}_{i=1}^\infty$ such that
\begin{enumerate}
	\item $\{p_i\}_{i=1}^\infty$ is closed;
	\item $D(\{p_i\}_{i=1}^\infty)$ is finite;
	\item $\overline{\mathcal{O}^-(p_i)}=\mathcal{M}$, for all $i\in\mathbb{N}$, and
	\item $\mathcal{O}(p_i)\cap\mathcal{O}(p_j)=\emptyset$, for every $i,j\in\mathbb{N}$, whenever $i\neq j$.
\end{enumerate}
As one sees from \cite{fs}, the set of those $\{p_i\}_{i=1}^\infty$ for which properties 3 and 4 hold, comprise a residual subset of $\ell_\infty(\mathcal{M})$.

\begin{lem} \label{dense-set-1}
	$\mathcal{S}^1$ is dense in $\ell_\infty(\mathcal{M})$.
\end{lem}

\begin{proof}
	Let $\{p_i\}_{i=1}^\infty\in\ell_\infty (M)$, and take $\epsilon>0$.
	
	Since $\mathcal{M}$ is compact, there exist $N$ points $x_1,\dots,x_N\in\mathcal{M}$ such that $\mathcal{M}=\bigcup_{i=1}^NB_{\epsilon/2}(x_i)$.
	
	One sees that, for any of these balls $B_{\epsilon/2}(x_i)$, we can perturb the subset $\{p_i\}_{i=1}^\infty\cap B_{\epsilon/2}(x_i)$ while staying within $B_{\epsilon/2}(x_i)$, in order to get a sequence $\{q_i\}_{i=1}^\infty\in\mathcal{S}^1$. To be precise, if $\{p_i\}_{i=1}^\infty\cap B_{\epsilon/2}(x_j)$ is infinite, than we can construct a subsequence of $\{q_i\}_{i=1}^\infty$ with only one accumulation point.
\end{proof}

\begin{lem} \label{recur1}
	For every $x\in\mathcal{I}$ and for every $\delta>0$ there exist a closed sequence $\{x_n\}_{n=1}^\infty\in\ell_\infty (M)$ contained in $B_\delta(x)$ and a sequence $\{r_n\}_{n=2}^\infty\subseteq\mathbb{R}_{>0}$ such that
	\begin{itemize}
		\item $x\in\{x_n\}_{n=1}^\infty$;
		\item $\overline{\mathcal{O}^-(x_i)}=\mathcal{M}$, for every $i \in \mathbb{N}$;
		\item $\mathcal{O}(x_i)\cap \mathcal{O}(x_j)=\emptyset$, whenever $i\neq j$, and
		\item $\overline{B_{r_n}(x_n)}\cap\overline{ B_{r_m}(x_m)}=\emptyset$, whenever $n\neq m$.
	\end{itemize}
\end{lem}

\begin{proof}
	Set $x_1=x$. For $n>2$, let us define $x_n$ and $r_n$ inductively. Since the set $\mathcal{I}\setminus(\bigcup_{i<n}\mathcal{O}(x_i))$ is dense, there exists$$x_n\in \bigg(B_{\delta/2^{n-1}}(x_1)\setminus \overline{B_{\delta/2^n}(x_1})\bigg)\cap\bigg(\mathcal{I}\setminus(\bigcup_{i<n}\mathcal{O}(x_i))\bigg).$$Set $r_n<\min\{\frac{\delta}{2^{n-1}}-d(x_n,x_1),d(x_n,x_1)-\frac{\delta}{2^n}\}$. Therefore, for any $n>2$, we have $B_{r_n}(x_n)\subseteq B_{\delta/2^{n-1}}(x_1)\setminus B_{\delta/2^n}(x_1)$.
\end{proof}

\begin{lem} \label{recur2}
	For every $x\in\mathcal{I}$, for every $\delta>0$ and for every $\alpha<\Omega$, there exists a closed sequence $\{x_n\}_{n=1}^\infty\in \ell_\infty (M)$ contained in $B_\delta(x)$ such that
	\begin{enumerate}
		\item $x\in\{x_n\}_{n=1}^\infty$;
		\item $\overline{\mathcal{O}^-(x_i)}=\mathcal{M}$, for every $i \in \mathbb{N}$;
		\item $\mathcal{O}(x_i)\cap \mathcal{O}(x_j)=\emptyset$, whenever $i\neq j$, and
		\item $D^\alpha(\{x_n\}_{n=1}^\infty)=\{x\}$.
	\end{enumerate}
\end{lem}

\begin{proof}
	Our goal is to define a closed countable set $A$ such that
	\begin{itemize}
		\item $x\in A$;
		\item $\overline{\mathcal{O}^-(q)}=\mathcal{M}$, for every $q \in A$;
		\item $\mathcal{O}(q)\cap \mathcal{O}(p)=\emptyset$, for every $p,q\in A$, whenever $p\neq q$, and
		\item $D^\alpha(A)=\{x\}$.
	\end{itemize}
	Let $A_0=\{x\}$. Using Lemma \ref{recur1}, we can construct a sequence $\{q_n\} _{n=1}^\infty$  and a sequence $\{r_n\}_{n=2}^\infty\subseteq\mathbb{R}_{>0}$ such that
	
	\begin{itemize}
		\item $x\in\{q_n\}_{n=1}^\infty$;
		\item $\overline{\mathcal{O}^-(q_i)}=\mathcal{M}$, for every $i \in \mathbb{N}$;
		\item $\mathcal{O}(q_i)\cap \mathcal{O}(q_j)=\emptyset$, whenever $i\neq j$, and
		\item $\overline{B_{r_n}(q_n)}\cap\overline{ B_{r_m}(q_m)}=\emptyset$, whenever $n\neq m$.
	\end{itemize}
	
	Now, for any ordinal $1<\lambda<\Omega$ let us define $A_\lambda$ recursively. 
	\begin{itemize}
		\item If $\lambda=\gamma+1$, then $A_\lambda =A_1\cup(\bigcup_{n=2}^\infty A_\gamma^n)$, where $A^n_\gamma$ is a closed sequence contained in $B_{r_n}(q_n)$ that satisfies the conditions \textit{1}, \textit{2} and \textit{3}, and such that $D^\gamma(A^n_\gamma)=\{q_n\}$. This follows from Lemma \ref{recur1}.
		\item If $\lambda$ is a limit ordinal, let $\{\gamma_n\}_{n=1}^\infty$ an increasing sequence of ordinals such that $\gamma_n<\lambda$ and $\sup\{\gamma_n\}_{n=1}^\infty= \lambda$. Then $A_\lambda =A_1\cup(\bigcup_{n=1}^\infty A_{\gamma_n}^n)$, where $A^n_{\gamma_n}$ is a closed sequence contained in $B_{r_n}(q_n)$ that satisfies the conditions 1, 2 and 3, and $D^{\gamma_n}(A^n_{\gamma_n})=\{q_n\}$.
	\end{itemize}
	Finally, set $A=A_\alpha$. Since $A$ is countable, we can enumerate this set as $\{x_n\}_{n=1}^\infty$.
\end{proof}
For every $K\in\mathbb{N}$ and for every ordinal $1<\beta<\Omega$, let $\mathcal{S}^\beta_K$ be the set of sequences $\{p_i\}_{i=1}^\infty$ such that
\begin{itemize}
	\item $\{p_i\}_{i=1}^\infty$ is closed;
	\item $\overline{\mathcal{O}^-(p_i)}=\mathcal{M}$, for all $i\in\mathbb{N}$;
	\item $\mathcal{O}(p_i)\cap\mathcal{O}(p_j)=\emptyset$, for every $i,j\in\mathbb{N}$ whenever $i\neq j$, and
	\item $\lvert D^\beta\{p_i\}_{i=1}^\infty\rvert=K$.
\end{itemize}
\begin{cor} \label{dense-set-2}
	For every $K\in\mathbb{N}$ and for every ordinal $1<\beta<\Omega$, $\mathcal{S}^\beta_K$ is dense in $\ell_\infty(\mathcal{M})$.
\end{cor}

As Example \ref{example} shows, the collection of closed sequences in $\ell_\infty(\mathcal{M})$ is not itself closed. Nonetheless, as one sees from Lemma \ref{dense-set-1} and Corollary \ref{dense-set-2}, closed sequences of all ranks are dense in $\ell_\infty(\mathcal{M})$.

Set $$\mathcal{S}=\mathcal{S}^1\cup\bigcup_{1<\beta<\Omega}\bigcup_{K\in\mathbb{N}}\mathcal{S}^\beta_K=\{\{p_i\}_{i=1}^\infty\in\ell_\infty(\mathcal{M}):\{p_i\}_{i=1}^\infty\text{ is closed}\}\text{.}$$ 

In the remainder of this section, we focus our attention on elements $\{p_i\}_{i=1}^\infty$ in $\mathcal{S}$. We begin with an observation concerning the sequences $\{\rho_n^i\}_{n\in\mathbb{N}} ^{i\in\mathbb{N}}$. In what follows, we take $\beta<\Omega$ to be the ordinal number such that $D^\beta(\{p_i\}_{i=1}^n)$ is finite and non-zero.

\begin{lem}\label{lem_rhos}
	Let $\{p_i\}_{i=1}^\infty$ be an element of $\mathcal{S}$. There exist radii $\{\rho_1^i\}^{i\in\mathbb{N}}$ such that	
	\begin{enumerate} 
		\item if $p_j,p_k\in D^\lambda (\{p_i\}_{i=1}^\infty)\setminus D^{\lambda+1}(\{p_i\}_{i=1}^\infty)$ for some ordinal $\lambda\leq\beta$, then $\overline{B_1^j}\cap\overline{B_1^k}=\emptyset$, and
		\item  if $p_j\in D^\lambda (\{p_i\}_{i=1}^\infty)\setminus D^{\lambda+1}(\{p_i\}_{i=1}^\infty)$ for some ordinal $\lambda<\beta$, then $\overline{B^j_1}\cap D^{\lambda+1}(\{p_i\}_{i=1}^\infty)=\emptyset$.
	\end{enumerate}
\end{lem}

\begin{proof}
	For every ordinal $\lambda<\beta$, let $\{p_{i_k}\}_{k=1}^\infty=D^\lambda (\{p_i\}_{i=1}^\infty)\setminus D^{\lambda+1}(\{p_i\}_{i=1}^\infty)$. For $k=1$, it is enough to take $\rho_1^{i_1}$ such that $$\rho^{i_1}_1<d\bigg(p_{i_1},\{p_{i_h}\}_{h=2}^\infty\cup D^{\lambda+1}(\{p_i\}_{i=1}^\infty)\bigg)\text{,}$$as the two sets are closed and disjoint.
	
	For $k>1$, it is enough to take $\rho^{i_k}_1$ such that $$\rho^{i_k}_1<\min\biggl\{d\bigg(p_{i_k},\{p_{i_h}\}_{h=k+1}^\infty\cup D^{\lambda+1}\{p_i\}_{i=1}^\infty\bigg),d\bigg(p_{i_k},\bigcup_{h=1}^{k-1}\overline{B_{\rho^{i_h}_1}(p_{i_h})}\bigg)\biggr\}\text{.}$$
	
	For $\lambda=\beta$, there are only finitely many points $p_{i_j}\in D^\beta(\{p_i\}_{i=1}^\infty)$, and the choice of the radii $\rho^{i_j}_1$ follows readily.
\end{proof}

\begin{cor}\label{cor1}
	Let $\{p_i\}_{i=1}^\infty\in\mathcal{S}$. For any collection of sequences $\{a^i_n\}_{n\in\mathbb{N}}^{i\in\mathbb{N}}\subseteq\mathbb{R}$ such that for every $i\in\mathbb{N}$,
	\begin{itemize}
		\item $0<a^i_{n+1}<a^i_n$, and
		\item $\lim_{n\longrightarrow\infty}a^i_n=0$,
	\end{itemize}
	there exists a sequence of tails  $\{\rho^i_n\}_{n\in\mathbb{N}}^{i\in\mathbb{N}}$ satisfying Lemma \ref{lem_rhos}. That is, for every $i\in\mathbb{N}$, $\{\rho^i_n\}_{n=1}^\infty=\{a^i_{N_i+n}\}_{n=1}^\infty$, for some $N_i\in\mathbb{N}$.
\end{cor}

\begin{proof}
	As in Lemma \ref{lem_rhos}, for every ordinal $\lambda<\beta$, let $\{p_{i_k}\}_{k=1}^\infty=D^\lambda (\{p_i\}_{i=1}^\infty)\setminus D^{\lambda+1}(\{p_i\}_{i=1}^\infty)$. For $k=1$, it is enough to take $N_{i_1}$ such that $$a^{i_1}_{N_{i_1}}<d\bigg(p_{i_1},\{p_{i_h}\}_{h=2}^\infty\cup D^{\lambda+1}(\{p_i\}_{i=1}^\infty)\bigg)\text{.}$$
	For $k>1$, it is enough to take $N_{i_k}$ such that $$a^{i_k}_{N_{i_k}}<\min\biggl\{d\bigg(p_{i_k},\{p_{i_h}\}_{h=k+1}^\infty\cup D^{\lambda+1}\{p_i\}_{i=1}^\infty\bigg),d\bigg(p_{i_k},\bigcup_{h=1}^{k-1}\overline{B_{a^{i_h}_{N_{i_h}}}(p_{i_h})}\bigg)\biggr\}\text{.}$$
	
	For $\lambda=\beta$, there are only finitely many points $p_{i_j}\in D^\beta(\{p_i\}_{i=1}^\infty)$. Therefore, one can choose the radii $N_{i_j}$ such that the closed balls $\overline{B^{i_j}_1}$ are pairwise disjoint.
\end{proof}

\begin{rem} \label{rem_rhos}
	Therefore, we can assume that, for every $i\in\mathbb{N}$, if $p_i\in D^\alpha(\{ p_k\}_{k=1}^\infty)\setminus D^{\alpha+1}(\{p_k\}_{k=1}^\infty)$, then for every $p_j\in \overline{B^i_1}$, there exists $\lambda<\alpha$ such that $p_j\in D^\lambda(\{p_k\}_{k=1}^\infty) \setminus D^{\lambda+1}(\{p_k\}_{k=1}^\infty) $. Furthermore, the second property assures that, if $p_j\in B_n^i$ for some $n\in\mathbb{N}$, then $B^j_m\subseteq B_{2\rho^i_n}(p_i)$ for all $m\in\mathbb{N}$.
\end{rem}


\subsection{Main results} \label{Results}

\begin{lem} \label{lem_main1}
Let $\{p_i\}_{i=1}^\infty$ be an element of $\mathcal{S}$, with $\{\rho_n^i\}_{n\in\mathbb{N}}^{i\in\mathbb{N}}$ as found in Lemma \ref{lem_rhos}. If $p_i\in\{p_j\}_{j=1}^\infty\setminus D(\{p_j\}_{j=1}^\infty)$, then $\mathfrak{T}(i)$ is residual.
\end{lem}

\begin{proof}
	Recall that $x\in \mathfrak{T}(i)$ if, for infinitely many natural numbers $n$, $\min\{k\in\mathbb{N}:f^k(x)\in B^i_n\}<\min\{k\in\mathbb{N}:f^k(x)\in B^j_n,~j\neq i\}$.
	Let
	\begin{equation}\label{C_n}
		C_n=\bigcup_{m\in\mathbb{N}}\big(f^{-m}(B^i_n)\setminus\big(\bigcup_{h=1}^m\bigcup_{j\neq i}f^{-h}(B^j_n)\big)\big),
	\end{equation}
	so that $\mathfrak{T}(i)=\bigcap_{M\in\mathbb{N}}\bigcup_{n>M}C_n$. Therefore, it is enough to show that $\bigcup_{n>M}C_n$ contains a dense and open set. We proceed by showing that, for every $m\in\mathbb{N}$, there exists a large enough $n$ such that $C_n$ contains an open neighborhood of $f^{-m}(p_i)$.
	
	Let $D=D(\{p_j\}_{j\neq i})=D( \{p_j\}_{j=1}^\infty)$. Since $D$ is closed, the set $\bigcup_{h=0}^mf^{-h}(D)$ is closed, too. Then, since $\mathcal{O}(p_i)\cap\mathcal{O}(p_j)=\emptyset$ whenever $j\neq i$, we can define$$\sigma=d\big(f^{-m}(p_i),\bigcup_{h=0}^mf^{-h}(D)\big).$$The family of functions $\{f^{-1},\dots,f^{-m}\}$ is equicontinuous, because $f$ is a homeomorphism on a compact space. Thus, there exists a number $\delta>0$ such that for any $x,y\in\mathcal{M}$ with $d(x,y)<\delta<\sigma/2$, it follows that $d(f^{-h}(x),f^{-h}(y))<\sigma/2$, for all $h=1,\dots,m$.
	
	For every $p_j\in D$, there exists a natural $n(j)$ such that $B^j_{n(j)}\subseteq B_{\delta/2}(D)$.
	Thus$$\bigcup_{p_j\in D}B^j_{n(j)}\supseteq D$$ is an open cover of the compact set $D$. So, there exists a set $\{p_{j_1},\dots,p_{j_K}\}$ such that $$D\subseteq \bigcup_{k=1}^KB^{j_k}_{n(j_k)}=: B.$$Let $N_1=\max\{n(j_1),\dots,n(j_K)\}$. Since $B$ is an open neighborhood of $D$, $B$ contains all but finitely many points of $\{p_i\}_{i=1}^\infty$. Moreover, thanks to property 2 of Lemma \ref{lem_rhos}, and thanks to Remark \ref{rem_rhos}, if $p_j\in B$, then $B^j_{N_1}\subseteq B_\delta(D)$. Therefore,$$\bigcup_{h=0}^m\bigcup_{p_j\in B}f^{-h}(B^j_n)\subseteq\overline{B_{\sigma/2}(\bigcup_{h=0}^mf^{-h}(D))}\quad\quad\text{for every } n>N_1\text{.}$$
	
	Let $\{p_{j_1'},\dots,p_{j_{K'}'}\}=\{p_j\}_{j\neq i}\setminus B$. Since there are finitely many such points, and $\mathcal{O}(p_i)\cap\mathcal{O}(p_j)=\emptyset$ whenever $j\neq i$, we can find a natural number $N_2$ such that$$f^{-m}(p_i)\notin\bigcup_{h=0}^m\bigcup_{k=1}^{K'}f^{-h}(\overline{B^{i_k'}_n})\quad\quad\text{for every }n>N_2.$$
	
	Hence, for every $n>\max\{N_1,N_2\}$,$$\bigcup_{h=0}^m\bigcup_{j\neq i}f^{-h}(B^j_n)\subseteq\overline{B_{\sigma/2}(\bigcup_{h=0}^mf^{-h}(D))}\cup\bigg(\bigcup_{h=0}^m\bigcup_{k=1}^{K'}f^{-h}(\overline{B^{j_k'}_n})\bigg).$$
	Since $$f^{-m}(p_i)\notin\overline{B_{\sigma/2}(\bigcup_{h=0}^mf^{-h}(D))}\cup\bigg(\bigcup_{h=0}^m\bigcup_{k=0}^{K'}f^{-h}(\overline{B^{j_k'}_n})\bigg),$$ there is an open neighborhood of $f^{-m}(p_i)$ disjoint from that set.
\end{proof}

\begin{lem} \label{lem_main2}
Let $\{p_i\}_{i=1}^\infty$ be an element of $\mathcal{S}$, with $\{\rho_n^i\}_{n\in\mathbb{N}}^{i\in\mathbb{N}}$ as found in Lemma \ref{lem_rhos}. If $p_i\in D(\{p_j\}_{j=1}^\infty)$, then $\mathfrak{T}(i)$ is a residual set.
\end{lem}

\begin{proof}
Let $C_n$ be as established in \eqref{C_n}.
We will show that for every $m\in\mathbb{N}$, there exist a sufficiently large natural number $n$ and an open set $G\subset C_n$ such that $f^{-m}(p_i)\in\partial(G)$, where $\partial(G)$ denotes the boundary of $G$.

Since $p_i$ is not a periodic point, we can define$$\sigma=\min\{d(f^{-k}(p_i),f^{-h}(p_i)):0\leq k,h\leq m,~k\neq h\}.$$As before, there exists a $\delta>0$ such that if $d(x,y)<\delta<\sigma/2$, it follows that $d(f^{-h}(x),f^{-h}(y))<\sigma/2$, for $h=0,\dots,m$. Therefore, the sets $\{f^{-h}(B_\delta(p_i))\}_{h=0,\dots,m}$ are pairwise disjoint. Let $\Bar{n}$ be such that $B^i_{\Bar{n}}\subseteq B_\delta(p_i)$.

Let $T=\{p_j\}_{j=1}^\infty\setminus B^i_{\Bar{n}}$. As in the previous proof, we can find a natural number $N_1>\Bar{n}$ such that, for every $n>N_1$, there is an open subset $G'$ of $f^{-m}(B^i_n)\setminus\bigcup_{h=0}^m\bigcup_{p_j\in T}f^{-h}(B^j_n)$ such that $f^{-m}(p_i)\in G'$.

Let $S=\{p_j\}_{j=1}^\infty\cap \overline{B^i_{\Bar{n}}}$.
By the definition of $\delta$ and $\Bar{n}$, we have that$$f^{-m}(B^i_n)\setminus\bigcup_{h=0}^m\bigcup_{p_j\in S\setminus\{p_i\}}f^{-h}(B^j_n)=f^{-m}(B^i_n)\setminus\bigcup_{p_j\in S\setminus\{p_i\}}f^{-m}(B^j_n)\text{,}\quad\quad\text{for every } n>\Bar{n}.$$
Since both $S$ and $\overline{B^i_{\Bar{n}}}$ are closed, we take $$\epsilon=d_H(S,\overline{B^i_{\Bar{n}}})>0\text{.}$$For every $p_j\in S\setminus \{p_i\}$ there exist a $n(j)$ such that $\rho^j_{n(j)}<\frac{\epsilon}{4}$. Thus,$$B_{\epsilon/4}(p_i)\cup\bigcup_{p_j\in S\setminus\{p_i\}}B^j_{n(j)}\subseteq B_{\epsilon/4}(S)$$is an open cover of the compact set $S$. Therefore, there is a finite set $\{p_{j_1},\dots,p_{j_K}\}$ such that$$S\subseteq B_{\epsilon/4}(p_i)\cup\bigcup_{k=1}^KB^{j_k}_{n(j_k)}.$$Let $N_2=\max\{n(j_1),\dots,n(j_K)\}.$ As before, for every $n>N_2$, $\bigcup_{p_j\in S\setminus\{p_i\}}B^j_n\subseteq B_{\epsilon/2}(S).$ Furthermore, since $p_i\notin \overline{B^j_n}$ for every $p_j\in S\setminus\{p_i\}$, we have $p_i\notin\bigcup_{p_j\in S\setminus\{p_i\}}\overline{B^j_n}$, but $p_i\in\overline{\bigcup_{p_j\in S\setminus\{p_i\}}B^j_n}.$ Therefore, $p_i\in\partial(\overline{\bigcup_{p_j\in S\setminus\{p_i\}}B^j_n})$. Hence, for every $n>N_2$, the set $$G''=f^{-m}\bigg(B^i_n\setminus\overline{\bigcup_{p_j\in S\setminus\{p_i\}}B^j_n}\bigg)\subseteq f^{-m}(B^i_n)$$ is a non-empty open set, because $\overline{\bigcup_{p_j\in S\setminus\{p_i\}}B^j_n}\subseteq \overline{B_{\epsilon/2}(S)}$ and $\overline{B^i_{\Bar{n}}}\not\subseteq\overline{B_{\epsilon/2}(S)}$. Moreover, $f^{-m}(p_i)\in\partial(G'')$.

Finally, for every $n>\max\{N_1,N_2\}$, the set$$G=G'\cap G''$$ is an open subset of $f^{-m}(B^i_n)\setminus\bigcup_{h=0}^m\bigcup_{j\neq i}f^{-h}(B^j_n)$ such that $f^{-m}(p_i)\in\partial(G)$.
\end{proof}

\begin{thm}
	Let $\{p_i\}_{i=1}^\infty$ be an element of $\mathcal{S}$, with $\{\rho_n^i\}_{n\in\mathbb{N}}^{i\in\mathbb{N}}$ as found in Lemma \ref{lem_rhos}. For every $\{p_i\}_{i=1}^\infty\in\mathcal{S}$, the open dynamical system $(\mathcal{M},d,f,\mathfrak{S})$ is completely indecisive.
\end{thm}

\begin{proof}
By Lemma \ref{lem_main1} and Lemma \ref{lem_main2}, $\mathfrak{T}(i)$ is residual for every $i\in\mathbb{N}$. It follows that $\mathfrak{T}=\cap_{i=1}^\infty \mathfrak{T}(i)$ is residual, too.
\end{proof}

Let $\mathcal{S}'$ be the set of sequences $\{p_i\}_{i=1}^\infty$ such that
\begin{itemize}
	\item Cl$(\{p_i\}_{i=1}^\infty)$ is countable;
	\item $\overline{\mathcal{O}^-(p_i)}=\mathcal{M}$, for all $i\in\mathbb{N}$;
	\item $\mathcal{O}(p_i)\cap\mathcal{O}(p_j)=\emptyset$, for every $i,j\in\mathbb{N}$, whenever $i\neq j$, and
	\item $\mathcal{O}(p_i)\cap \mathcal{O}(\text{Cl}(\{p_j\}_{j=1}^\infty)\setminus\{p_i\})=\emptyset$ for every $i\in\mathbb{N}$.
\end{itemize}
\begin{cor}
	Let $\{p_i\}_{i=1}^\infty$ be an element of $\mathcal{S'}$, with $\{\rho_n^i\}_{n\in\mathbb{N}}^{i\in\mathbb{N}}$ as found in Lemma \ref{lem_rhos}.	
	The open dynamical system $(\mathcal{M},d,f,\mathfrak{S})$ is completely indecisive.
\end{cor}
\begin{proof}
	The set $\mathcal{I}'=\{\{p_i\}_{i=1}^\infty\in\mathcal{I}:\mathcal{O}(p_i)\cap\mathcal{O}(\text{Cl}(\{p_j\}_{j=1}^\infty)\setminus\{p_i\})=\emptyset\text{, for every }i\in\mathbb{N}\}$ is residual, as shown in \cite{fs}. Thus, following the proofs found in Lemma \ref{dense-set-1}, \ref{recur1}, \ref{recur2} and \ref{dense-set-2}, and using $\mathcal{I}'$ instead of $\mathcal{I}$, we have the desired statement.
\end{proof}

\section{Cl$(\{p_i\}_{i=1}^\infty)$ is nowhere dense}\label{nwd-closure}
In this section, we consider the case in which Cl$(\{p_i\}_{i=1}^\infty)$ is nowhere dense, regardless of whether it is countable or not. Let $\mathcal{T}$ denote the set of sequences $\{p_i\}_{i=1}^\infty$ such that $\mathcal{O}(p_i)\cap\mathcal{O}(\text{Cl}(\{p_j\}_{j=1}^\infty)\setminus\{p_i\})=\emptyset$, so that
\begin{itemize}
	\item $\mathcal{O}(p_i)\cap\mathcal{O}(p_j)=\emptyset$, whenever $i\neq j$;
	\item $\mathcal{O}(p_i)\cap\mathcal{O}(D(\{p_j\}_{j=1}^\infty)\setminus\{p_i\})=\emptyset$, for every $i\in\mathbb{N}$, and
	\item $D(\{p_j\}_{j=1}^\infty)$ is nowhere dense.
\end{itemize}
Consider also the set $\mathcal{A}$ of sequences such that
\begin{itemize}
	\item $\mathcal{O}(p_i)\cap\mathcal{O}(p_j)=\emptyset$, whenever $i\neq j$;
	\item $\mathcal{O}(p_i)\cap\mathcal{O}(D(\{p_j\}_{j=1}^\infty))=\emptyset$, for every $i\in\mathbb{N}$, and
	\item $\{p_i\}_{i=1}^\infty\cap D(\{p_i\}_{i=1}^\infty)=\emptyset$.
\end{itemize}
Thanks to Lemmas 7, 8 and 9 of \cite{fs}, we know that $\mathcal{A}$ is a residual set. We also have that $\mathcal{A}\subseteq\mathcal{T}$. Thus, $\mathcal{T}$ is residual, too.

\begin{lem} \label{rhos2}
	Let $\{p_i\}_{i=1}^\infty$ be an element of $\mathcal{T}$. For every $i\in\mathbb{N}$, we can take $B^i_1$ such that $p_j\notin\overline{B^i_1}$ for every $1\leq j<i$.
\end{lem}
\begin{proof}
	It is enough to take $\rho^i_1<d(p_i,\{p_1,\dots,p_{i-1}\}).$
\end{proof}
In the sequel, we consider a sequence of sequences $\{\rho^i_n\}_{n\in\mathbb{N}}^{i\in\mathbb{N}}$ that satisfies the property of Lemma \ref{rhos2}, and such that $\sup_i\rho^i_n\xrightarrow{n\longrightarrow\infty}0$.

As with Corollary \ref{cor1}, we first show that an appropriate collection of radii $\{\rho_n^i\}_{n\in\mathbb{N}}^{i\in\mathbb{N}}$ can be extracted from the tails of any set of infinitesimal sequences $\{a_n^i\}_{n\in\mathbb{N}}^{i\in\mathbb{N}}$.

\begin{cor}
	Let $\{p_i\}_{i=1}^\infty\in\mathcal{T}$. For any collection of sequences $\{a^i_n\}_{n\in\mathbb{N}}^{i\in\mathbb{N}}\subseteq\mathbb{R}$ such that for every $i\in\mathbb{N}$,
	\begin{itemize}
		\item $0<a^i_{n+1}<a^i_n$, and
		\item $\lim_{n\longrightarrow\infty}a^i_n=0$,
	\end{itemize}
	there exists a sequence of tails $\{\rho^i_n\}_{n\in\mathbb{N}}^{i\in\mathbb{N}}$, such that $p_j\in\overline{B^1_i}$ whenever $1\leq j<i$, and $\sup_i \rho^i_n\longrightarrow0$ as $n\longrightarrow\infty$. That is, for every $i\in\mathbb{N}$, $\{\rho^i_n\}_{n=1}^\infty=\{a^i_{N_i+n}\}_{n=1}^\infty$, for some $N_i\in\mathbb{N}$.
\end{cor}

\begin{proof}
	For every $i\in\mathbb{N}$ it is enough to take $N_i\in\mathbb{N}$ such that
	$$\rho^i_1=a^i_{N_i+1}<\min\biggl\{d(p_i,\{p_1,\dots p_{i-1}\}),\frac{1}{i}\biggr\}\text{.}$$
	Thus, $\rho^i_n<1/i$ for every $n\in\mathbb{N}$, and for every $i\in\mathbb{N}$. We now show that this implies that $\sup_{i\in\mathbb{N}}\rho^i_n\xrightarrow{n\longrightarrow\infty}0$. Since $\sup_{i\in\mathbb{N}}\rho^i_{n+1}\leq\sup_{i\in\mathbb{N}}\rho^i_n$, it is enough to show that for any $\epsilon>0$, there exists $N\in\mathbb{N}$ such that $\sup_{i\in\mathbb{N}}
	\rho^i_N<\epsilon$. Let $i\in\mathbb{N}$ such that $1/i<\epsilon$. Then, $\rho^j_n<1/i<\epsilon$ for every $n\in\mathbb{N}$ and for every $j\geq i$. Since $\{\rho^1_n\}_{n=1}^\infty,\dots,\{\rho^{i-1}_n\}_{n=1}^\infty$ are finite, there exists $N\in\mathbb{N}$ such that $\rho^j_N<\epsilon$ for every $j=1,\dots,i-1$. Hence, $\sup_{i\in\mathbb{N}}\rho^i_N<\epsilon$.
\end{proof}

\begin{rem}
	Lemma \ref{rhos2} implies that, if $p_j\in B^i_1$, then $p_i\notin\overline{B^j_1}$.
\end{rem}
\begin{thm}
	Let $\{p_i\}_{i=1}^\infty\in\mathcal{T}$. Then, for every $i\in\mathbb{N}$, $\mathfrak{T}(i)$ is residual. Thus, $\mathfrak{T}$ is residual, and the open dynamical system $(\mathcal{M},d,f,\mathfrak{S})$ is completely indecisive.
\end{thm}
\begin{proof}
	Like in the earlier cases, we show that for every $m\in\mathbb{N}$ there exists a large enough $n\in\mathbb{N}$ such that $f^{-m}(B^i_n)\setminus\bigcup_{h=1}^m\bigcup_{j\neq i}f^{-h}(B^j_n)$ contains an open subset $G$, with $f^{-m}(p_i)\in\partial G$.
	
	Since $p_i$ is not a periodic point, let$$\sigma=\min\{d(f^{-k}(p_i),f^{-h}(p_i)):~0\leq k,h\leq m,~k\neq h\}.$$There exists a $\delta>0$ such that if $d(x,y)<\delta$, it follows that $d(f^{-h}(x),f^{-h}(y))<\sigma/2$, for $h=0,\dots,m$. Therefore, the sets $\{f^{-h}(B_\delta(p_i))\}_{h=0,\dots,m}$ are pairwise disjoint. Let $\Bar{n}$ be such that $B^i_{\Bar{n}}\subseteq B_\delta(p_i)$.
	
	Let us define $T=\text{Cl}(\{p_j\}_{j=1}^\infty)\setminus B^i_{\Bar{n}}$ and $S=\text{Cl}(\{p_j\}_{j=1}^\infty)\cap \overline{B^i_{\Bar{n}}}$. Both the set are closed. First, let us consider $T$. Since $\mathcal{O}(p_i)\cap\mathcal{O}(\text{Cl}(\{p_j\}_{j=1}^\infty)\setminus\{p_i\})=\emptyset$, we can define$$\epsilon=d(f^{-m}(p_i),\bigcup_{h=0}^mf^{-h}(T))>0.$$Since $\{f^{-1},\dots,f^{-m}\}$ are equicontinuous and $\sup_i\rho^i_n\xrightarrow{n\longrightarrow\infty}0$, there exists a natural number $N_1$ such that$$\bigcup_{h=0}^m\bigcup_{p_j\in T}f^{-h}(B^j_n)\subseteq B_{\epsilon/2}(\bigcup_{h=1}^mf^{-h}(T))\text{,}\quad\quad\text{for every } n>N_1\text{.}$$Therefore, for every $n>N_1$, there exists an open set $G'\subseteq f^{-m}(B^i_n)\setminus \bigcup_{h=0}^m\bigcup_{p_j\in T}f^{-h}(B^j_n)$ such that $f^{-m}(p_i)\in G'$.
	
	By definition of $\delta$ and $\Bar{n}$, we have that$$f^{-m}(B^i_n)\setminus\bigcup_{h=0}^m\bigcup_{p_j\in S\setminus\{p_i\}}f^{-h}(B^j_n)=f^{-m}(B^i_n)\setminus\bigcup_{p_j\in S\setminus\{p_i\}}f^{-m}(B^j_n)\text{,}\quad\quad\text{for every } n>\Bar{n}.$$Let $\eta$ be the Hausdorff distance between the closed sets $S$ and $\overline{B^i_{\Bar{n}}}$; that is, $\eta=d_H(S,\overline{B^i_{\Bar{n}}})$. Since $S\subseteq\overline{B^i_{\Bar{n}}}$ and $S$ is nowhere dense, we have $\eta>0.$ Since $f^{-m}$ uniformly continuous and $\sup_i\rho^i_n\xrightarrow{n\longrightarrow\infty}0$, there exists a natural number $N_2>\Bar{n}$ such that
	$$\bigcup_{p_j\in S\setminus\{p_i\}}B^j_n\subseteq B_{\eta/2}(S)\text{,}\quad\quad\text{for every } n>N_2.$$Furthermore, since $p_i\notin \overline{B^j_n}$ for every $p_j\in S\setminus\{p_i\}$, we have $p_i\notin\bigcup_{p_j\in S\setminus\{p_i\}}\overline{B^j_n}$, but $p_i\in\overline{\bigcup_{p_j\in S\setminus\{p_i\}}B^j_n}.$ Therefore, $p_i\in\partial(\overline{\bigcup_{p_j\in S\setminus\{p_i\}}B^j_n})$. Hence, for every $n>N_2$, the set $$G''=f^{-m}\bigg(B^i_n\setminus\overline{\bigcup_{p_j\in S\setminus\{p_i\}}B^j_n}\bigg)\subseteq f^{-m}(B^i_n)$$ is a non-empty open set, because $\overline{\bigcup_{p_j\in S\setminus\{p_i\}}B^j_n}\subseteq \overline{B_{\eta/2}(S)}$ and $\overline{B^i_{\Bar{n}}}\not\subseteq\overline{B_{\eta/2}(S)}$. Moreover, $f^{-m}(p_i)\in\partial(G'')$.
	
	Finally, for every $n>\max\{N_1,N_2\}$, the set$$G=G'\cap G''$$ is an open subset of $f^{-m}(B^i_n)\setminus\bigcup_{h=1}^m\bigcup_{j\neq i}f^{-h}(B^j_n)$, such that $f^{-m}(p_i)\in\partial(G)$.
	
\end{proof}

\section{Cl$(\{p_i\}_{i=1}^\infty)$ is somewhere dense}\label{swd_closure}
In this section we consider the case in which $\{p_i\}_{i=1}^\infty$ is somewhere dense.

\begin{lem}
	For every $\{p_i\}_{i=1}^\infty$ such that Cl$(\{p_i\}_{i=1}^\infty)$ is somewhere dense, and for every $\{\rho^i_n\}_{n\in\mathbb{N}}^{i\in\mathbb{N}}$, the open dynamical system $(\mathcal{M},d,f,\mathfrak{S})$ is not completely indecisive.
\end{lem}

\begin{proof}
	Take $\{p_i\}_{i=1}^\infty\in\ell_\infty(\mathcal{M})$. Suppose that there exists an open subset $B$ such that $B\subseteq$Cl$(\{p_i\}_{i=1}^\infty)$. Let $p_i\in B$. Thus, there exists $M\in\mathbb{N}$ such that $B^i_M\subseteq B$, and for every $n> M$, $B^i_n\subseteq B$. Fix $n>M$, and let $\mathfrak{C}_n=\{j\in\mathbb{N}:p_j\in B^i_n\}$. Then, $B^i_n\subseteq\text{Cl}(\bigcup_{j\in\mathfrak{C}_n}B^j_n)$, and $B^i_n\setminus\bigcup_{j\in\mathfrak{C}_n}B^j_n$ is nowhere dense. This is true for every $n>M$. Therefore,
	$$\bigcup_{n>M}\bigcup_{m\in\mathbb{N}}\bigg( f^{-m}(B^i_n)\setminus\bigcup_{h=0}^m\bigcup_{j\neq i} f^{-h}(B^j_n)\bigg)$$is of the first category.
	It follows that
	$$\mathfrak{T}(i)=\bigcap_{M\geq 1}\bigcup_{n> M}C_n$$ is of the first category, where $C_n=\bigcup_{m\in\mathbb{N}}\big( f^{-m}(B^i_n)\setminus\bigcup_{h=0}^m\bigcup_{j\neq i} f^{-h}(B^j_n)\big)$, and $\mathfrak{T}$ is not a residual set.
\end{proof}


\bibliographystyle{plain}
\bibliography{bib_nowhere_dense_holes}

\end{document}